\documentclass[12pt, reqno]{amsart}
\usepackage[margin=1in]{geometry}
\usepackage[english]{babel}
\usepackage[latin1]{inputenc}
\usepackage{amssymb}
\usepackage{amsmath}
\usepackage{arydshln}
\usepackage{booktabs}
\usepackage{graphicx}
\usepackage{epsfig}
\usepackage{enumitem}
\usepackage{hyperref}
\usepackage{ mathrsfs }
\usepackage{mathtools}
\usepackage[new]{old-arrows}
\usepackage[all]{xy}
\usepackage[usenames, dvipsnames]{color}

\usepackage{tikz}

\usepackage{graphicx}
\usepackage{caption}
\usepackage{subcaption}
\usepackage[]{hyperref}

\newcommand{\C}{\mathcal{C}}
\newcommand{\R}{\mathbb{R}}
\newcommand{\N}{\mathbb{N}}

\newcommand{\Ss}{\mathcal{S}}

\newtheorem{theorem}{Theorem}[section]
\newtheorem{lemma}[theorem]{Lemma}

\newtheorem{cor}[theorem]{Corollary}

\newtheorem{conj}{Conjecture}[section]
\newtheorem{op}{Open Problem}[section]

\theoremstyle{definition}
\newtheorem*{defi}{Definition}

\theoremstyle{remark}
\newtheorem{rem}[theorem]{Remark}
\newtheorem*{remark}{Remark}

\makeatletter
\let\c@table\c@figure % for (1)
\let\ftype@table\ftype@figure % for (2)
\makeatother

\title[Spectral gaps of Robin Schr\"odinger operators with single-well potentials]{Spectral gaps of 1-D Robin Schr\"odinger operators with single-well potentials}
\author{Mark S. Ashbaugh and Derek Kielty}

\address{Department of Mathematics, University of Illinois, Urbana, IL 61801, U.S.A.}
\email{dkielty2@illinois.edu}

\address{Department of Mathematics
University of Missouri
Columbia, MO 65211, U.S.A.}
\email{ashbaughm@missouri.edu}

\begin{document}
\maketitle

\begin{abstract}
We prove sharp lower bounds on the spectral gap of 1-dimensional Schr\"odinger operators with Robin boundary conditions for each value of the Robin parameter. In particular, our lower bounds apply to single-well potentials with a centered transition point. This result extends work of Cheng et al.\ and  Horv\'ath in the Neumann and Dirichlet endpoint cases to the interpolating regime. We also build on recent work by Andrews, Clutterbuck, and Hauer in the case of convex and symmetric single-well potentials. In particular, we show the spectral gap is an increasing function of the Robin parameter for symmetric potentials.
\end{abstract}

%%%%%%%%%%%%%%%%%%%%%%%%%%%%%%%%%%%%%%%%%%%%%%%%%%%%%%%%%%%%%%%%%%%%%%%%%%%%%%%%

\section{\textbf{Introduction}}

The Schr\"odinger eigenvalue problem with Robin boundary conditions and parameter $\alpha$,
\begin{equation}
\label{eq:RobinProb}
\begin{cases}
      -u'' + Vu = \lambda u, \quad \text{on} ~ (-L/2,L/2) \\
      u'(-L/2) = \alpha u(-L/2) ~ \text{and} ~ u'(L/2) = -\alpha u(L/2),
\end{cases}
\end{equation} 
has a discrete spectrum of simple eigenvalues
\[\lambda_1 < \lambda_2 < \lambda_3 < \dots \to \infty.\]
In what follows $V : [-L/2,L/2]
\to \R$ is a bounded potential and $\alpha \in (-\infty,\infty]$. We use the convention that $\alpha = \infty$ means Dirichlet boundary conditions, $u(-L/2) = 0 = u(L/2)$. Notice that $\alpha = 0$ is the Neumann boundary condition. Let \[\Lambda(V,\alpha) = \lambda_2(V,\alpha) - \lambda_1(V,\alpha)\] 
denote the fundamental spectral gap of (\ref{eq:RobinProb}). In this paper we prove a family of sharp lower bounds on the spectral gap for each $\alpha$ in a certain range. Our main result concerns the class of single-well potentials (see Figure \ref{fig:CSW}).
\begin{defi}
A potential function $V$ is a \emph{single-well potential} on the closed interval $I$ if there is a $\tau \in I$ such that $V$ is nonincreasing for $x \leq \tau$ and nondecreasing for $x \geq \tau$. Such a number $\tau$ is called a \emph{transition point}.
\end{defi}

\begin{figure}
    \includegraphics[scale = .7]{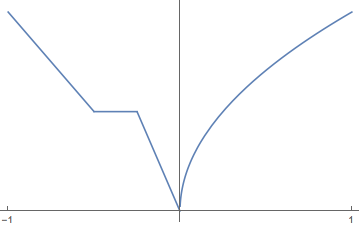}
    \caption{Single-well potential with transition point $\tau = 0$.}
    \label{fig:CSW}
\end{figure}

The following result generalizes work of Cheng et al.\ \cite[Theorem 1.1]{CKLL} for Neumann ($\alpha = 0$) and Horv\'ath \cite[Theorem 1.1]{Horvath} for Dirichlet boundary conditions ($\alpha = \infty$) to $\alpha \in [0,\infty]$.
\begin{theorem}
\label{thm:Horvath}
Let $V$ be a single-well potential. If $V$ has a centered transition point $\tau = 0$ then
\[\Lambda(V,\alpha) \geq \Lambda(0,\alpha), \quad \text{for each } \alpha \in [0,\infty]\]
with equality if and only if $V$ is constant. Moreover, for each $\alpha \in (-\infty,\infty]$ there are $V$ with $\tau \neq 0$ such that $\Lambda(V,\alpha) < \Lambda(0,\alpha)$. 
\end{theorem}

\noindent The lower bound $\Lambda(0,\alpha)$ is a strictly increasing function of $\alpha$. Moreover, $\Lambda(0,\alpha) \to 0$ as $\alpha \to -\infty$, $\Lambda(0,0) = \pi^2/L^2$, and $\Lambda(0,\infty) = 3\pi^2/L^2$ (see \cite[Lemmas 5.1 and 5.2]{L}).

In the case of symmetric single-well potentials Andrews, Clutterbuck, and Hauer recently extended the lower bounds on the spectral gap for Neumann and Dirichlet boundary conditions to Robin boundary conditions for the $p$-Laplacian with a symmetric single-well potential in \cite[Theorem 1.1]{ACH2}. In the case $p=2$, the following theorem expands on their results by allowing for a symmetric background potential and a varying boundary parameter. It also generalizes work of Ashbaugh and Benguria \cite{AB} from the Dirichlet case ($\alpha = \infty$) to $\alpha \in (-\infty,\infty]$. 
\begin{theorem}
\label{thm:SWS}
If $S$ is a symmetric potential, $V$ is symmetric single-well, $\alpha \in (-\infty,\infty]$, and $\gamma \geq 0$ then
\[\Lambda(S + V,\alpha + \gamma) \geq \Lambda(S,\alpha),\]
with equality if and only if $V$ is constant and $\gamma = 0$.
\end{theorem}

Taking $V = 0$ in the above theorem implies monotonicity of the gap with respect to the boundary condition.
\begin{cor}
\label{cor:GapInc}
If $S$ is a symmetric potential then $\Lambda(S,\alpha)$ is a strictly increasing function of $\alpha \in (-\infty,\infty]$.
\end{cor}

In the case of convex potentials Andrews, Clutterbuck, and Hauer \cite[Theorem 1.4]{ACH2} proved lower bounds on the gap for $\alpha \in [-1/L,\infty]$ that generalizes work of Lavine \cite{Lav}. We generalize this result to asymmetric Robin boundary conditions. Let $\Lambda(V,(\alpha,\beta))$ denote the spectral gap of the Schr\"odinger operator with Robin boundary conditions $u'(-L/2) = \alpha u(-L/2)$ and $u'(L/2) = -\beta u(L/2)$ replacing the boundary condition in (\ref{eq:RobinProb}).

\begin{theorem}
\label{thm:VConvNew}
If $V$ is convex then \[\Lambda(V,(\alpha,\beta)) \geq \Lambda(0,\min\{\alpha,\beta\}), \quad \text{for each} ~ \alpha,\beta \in [-1/L,\infty],\] 
with equality if $V$ is constant and $\alpha = \beta$.
\end{theorem}
Letting $\alpha = \beta$ we recover Theorem 1.4 in \cite{ACH2}, which says: among convex potentials $\Lambda(V,\alpha)$ is minimal for the constant potential. In contrast to this result, the convex potential that minimizes $\Lambda(V,(\alpha,\beta))$ may not be the constant potential when $\alpha \neq \beta$ (see Section \ref{sec:OP}).

The above results hold for bounded potentials, but one can extend them to potentials in $L^1(-L/2,L/2)$ by an approximation argument at the cost of the ``only if" equality statements in Theorems \ref{thm:Horvath} and 
\ref{thm:SWS}. The results stated so far have only addressed single-well potentials, specifically, convex potentials and single-well potentials with symmetry assumptions. For symmetric single-well potentials it is possible to derive ``reverse forms" of the above inequalities for single-barrier potentials by using similar arguments. Single barrier potentials are those that first increase and then decrease.

\section{\textbf{Literature}}

There is a significant literature on spectral gaps of the Laplacian and Schr\"odinger operators, especially with Neumann and Dirichlet boundary conditions. For an overview of diameter lower bounds on the fundamental spectral gap see the write-up by Ashbaugh \cite{AshAIM} from the AIM workshop \cite{AIM1}. In particular, there are various results known in one dimension besides those mentioned in the previous section. Notably, Harrell and El Allali proved the novel lower bound $\Lambda(V,\infty) \geq \theta^2 \pi^2/L^2$, where $\theta^2 \approx 2.04575$ ($x = \theta \pi$ is the first positive root of $\tan(x) = x$) when $V$ is a general single-well potential, that is, one without convexity or symmetry assumptions \cite[Theorem 3.1]{HZ}. One should compare this with the lower bound of $3 \pi^2/L^2$ for the potentials in the Dirichlet cases of Theorems and \ref{thm:Horvath}, \ref{thm:SWS}, and \ref{thm:VConvNew}. Still in one dimension Abramovich generalized the result of Ashbaugh and Benguria to double-well potentials with symmetry conditions \cite{A}.

In higher dimensions the first lower bound on the spectral gap is due to Payne and Weinberger who showed that the spectral gap of the Neumann Laplacian on a bounded domain $\Omega$ is bounded from below by $\pi^2/D^2$, where $D$ is the diameter of $\Omega$ \cite{PW}. More recently Andrews and Clutterbuck proved that Dirichlet Schr\"odinger operators on a convex domain with a convex potential have a fundamental spectral gap that is always larger than $3 \pi^2 /D^2$ \cite{AC}. Less is known about the Robin Laplacian, even considering inequalities on other spectral quantities. For an introduction to the Robin Laplacian and an overview of what is known for it see Chapter 4 (by Bucur, Freitas, and Kennedy) of the book \cite{Henrot}. For interesting conjectures see Henrot \cite{Henrot} (as above) and Laugesen \cite{L}.

\section{\textbf{Open problems and questions}}
\label{sec:OP}

\subsection*{Lower bounds on gaps for \texorpdfstring{$\alpha < 0$}{Lg}:} The lower bound in Theorem \ref{thm:SWS} holds for $\alpha \in (-\infty,\infty]$, but the lower bounds in Theorems \ref{thm:VConvNew} and \ref{thm:Horvath} only hold for $\alpha \geq -1/L$ and $\alpha \geq 0$, respectively. Nonetheless, we expect the following conjecture to hold:
\begin{conj}
\label{conj:AllAlpha}
If $V$ is single-well with centered transition point $\tau = 0$ or convex then 
\[\Lambda(V,\alpha) \geq \Lambda(0,\alpha), \quad \text{for each} ~ \alpha \in (-\infty,\infty],\]
with equality if and only if $V$ is constant.
\end{conj}
\noindent When $V$ is a single-well potential with centered transition point the conjecture is supported by numerical calculations (see Section \ref{sec:Num} and Figures \ref{fig:NegGapM} and \ref{fig:PosGapM}).

\iffalse
, \textcolor{red}{but it now known that when $\alpha < 0$ there is no lower bound on the gap \cite{Kielty}. The phenomenon responsible for the failure of the conjecture in higher dimensions is the concentration of the Robin ground state around \emph{two} points as a sequence of domains degenerates. We believe this concentration phenomenon can not be replicated by a sequence of single-well potentials in one dimension, which further supports Conjecture \ref{conj:AllAlpha}.}
\fi

\subsection*{Asymmetric Robin conditions:} In light of Theorem \ref{thm:VConvNew}, it is natural to consider the problem: 
\begin{op}
Let $\alpha,\beta \in (-\infty,\infty]$ with $\alpha \neq \beta$. Determine the minimizers of $\Lambda(V,(\alpha,\beta))$ among convex potentials $V$. Equivalently, determine the number $a \in \R$ that minimizes $\Lambda(ax,(\alpha,\beta))$ as a function of $\alpha$ and $\beta$.
\end{op}

When $\alpha = \beta \geq -1/L$ Theorem \ref{thm:VConvNew} shows the minimizer is the constant potential, but when $\alpha \neq \beta$ we expect that the minimizer is a \emph{non-constant} linear potential. In particular, when $\alpha = \infty$ and $\beta = 0$ we show the minimizer is of the form $V(x) = ax$ for some $a \neq 0$.

To see this, note that the first part of the proof of Theorem \ref{thm:VConvNew} shows that the minimizer among convex potentials exists and is a linear potential. Next observe that when $V \equiv 0$ the $L^2$-normalized eigenfunctions are $u_1(x) = C\sin(\pi (x + L/2)/2L)$ and $u_2(x) = C\sin(3\pi (x + L/2)/2L)$ where $C^2 = 2/L$. Finally, Lemma \ref{lem:Deriv} shows that
\[\frac{d}{da} \bigg|_{a = 0}\Lambda(ax,(\infty,0)) =  C^2\int_{-L/2}^{L/2} x (u_2(x)^2 - u_1(x)^2) \, dx < 0,\]
so that $\Lambda(ax,(\infty,0)) < \Lambda(0,(\infty,0)) = 2 \pi^2/L^2$ for small $a > 0$. A similar calculation would likely show that the minimizer is also a non-constant linear potential for general $\alpha \neq \beta$. 

An analogous argument also shows that among single-well potentials with centered transition point (as in Theorem \ref{thm:Horvath}) the minimizer is a step function of the form $m \mathbf{1}_{(0,L/2)}(x)$ for some $m \neq 0$ when the mixed Dirichlet--Neumann boundary conditions are imposed. In a similar direction, Harrell and El Allali also observed when attempting to minimize $\Lambda(V_0 + V,\infty)$ over convex $V$ for the asymmetric background potential $V_0(x) = x$, the minimizer is $-x + b$ rather than a constant potential \cite[p.\ 13]{HZ}.

\subsection*{Gap monotonicity:} Corollary \ref{cor:GapInc} shows that $\alpha \mapsto \Lambda(V,\alpha)$ is an increasing function when $V$ is symmetric. One would also like to know if monotonicity holds for asymmetric potentials, especially single-well and convex ones. This leads us to the following problem: 
\begin{op}
\label{op:mono}
Determine a class of potentials $V$ for which $\alpha \mapsto \Lambda(V,\alpha)$ is monotone for $\alpha \geq 0$.
\end{op}
\noindent This problem is only stated for $\alpha \geq 0$ because numerical computation shows $\alpha \mapsto \Lambda(m \boldsymbol{1}_{(0,\pi/2)},\alpha)$ is not monotone for $\alpha < 0$ (see Figure \ref{fig:GapAlphaPlot} in Section \ref{sec:Num}).

Interestingly, Smits observed in one dimension that even the higher gaps (without potential) $\alpha \mapsto (\lambda_j - \lambda_i)(0,\alpha)$ are  increasing functions for $\alpha > 0$ whenever $j > i$ \cite[\S 4]{Smits}. This suggests that Corollary \ref{cor:GapInc} may hold for higher gaps as well. 

Corollary \ref{cor:GapInc} is similar in spirit to a conjecture due to Smits in higher dimensions. To state the conjecture let $\lambda^{\Omega}_j(\alpha)$ for $j = 1,2,3, \dots$ denote the Robin eigenvalues of $-\Delta$ with parameter $\alpha$ on a bounded Lipschitz domain $\Omega$ and $\Lambda^{\Omega}(\alpha) = (\lambda^{\Omega}_2 - \lambda^{\Omega}_1)(\alpha)$ its fundamental spectral gap.

\begin{conj}[Smits \protect{\cite[\S 4]{Smits}}]
If $\Omega$ is a convex bounded domain then $\Lambda^{\Omega}(\alpha)$ is an increasing function for $\alpha \geq 0$.
\end{conj}
The conjecture is known for intervals and disks by Smits \cite{Smits} and for rectangular boxes (i.e.\ product of intervals) for $\alpha \in \R$ in all dimensions by Laugesen \cite[Theorem 2.1]{L}. In two dimensions monotonicity holds for general $\Omega$ at the endpoints with the uniform bounds:
\[\Lambda^{\Omega}(0) < 4j_0^2/D^2  < 3\pi^2/D^2 \leq \Lambda^{\Omega}(\infty).\] 
The upper bound on the Neumann gap $\Lambda^{\Omega}(0)$ follows from work of Ba{\~n}uelos and Burdzy \cite[Corollary 2.1]{BB}, where $j_0 \approx 2.4048$ is the first positive root of the Bessel function $J_0$. The lower bound on the Dirichlet gap $\Lambda^{\Omega}(\infty)$ is due to Andrews and Clutterbuck \cite[Corollary 1.4]{AC}. In dimensions three and larger there are $\Omega$ for which $\Lambda^{\Omega}(0) > 3 \pi^2/D^2$ so this reasoning cannot prove $\Lambda^{\Omega}(0) < \Lambda^{\Omega}(\infty)$ in higher dimensions.

\section{\textbf{Preliminary lemmas}}
For notational ease let $I = (-L/2,L/2)$. We will denote weak eigenfunctions of (\ref{eq:RobinProb}) with Robin parameters $\alpha$ and $\beta$ at $x = \pm L/2$ corresponding to $\lambda_j$ by $u_j \in H^1(I)$. Without loss of generality, we assume that each $u_j$ is $L^2$-normalized, $u_1 > 0$, and $u_2$ is positive near $x = -L/2$ by multiplying (by $-1$ if necessary). Let $x_0 \in I$ denote the unique zero of $u_2$. The following lemma is well-known and the basic argument for it can be found in \cite{AB} for Dirichlet boundary conditions.
\begin{lemma}
\label{lem:Wrskn}
Let $\alpha,\beta \in (-\infty,\infty]$. If $V \in L^{\infty}(I)$ then there exist points $-L/2 \leq x_- < x_0 < x_+\leq L/2$ such that $u_2^2 > u_1^2$ on a set of the form $(-L/2,x_-) \cup (x_+,L/2)$ and $u_2^2 < u_1^2$ on $(x_-,x_+)$. If $V$ is also symmetric then $x_+ \in (0,L/2)$, $x_- = - x_+$, and $u_2^2 > u_1^2$ on $(-L/2,-x_+) \cup (x_+,L/2)$ with $u_2^2 < u_1^2$ on $(-x_+,x_+)$.
\end{lemma}

\begin{proof}
We will show that $u_2/u_1$ is a decreasing function so that the set $\{u_2^2 > u_1^2\} = \{(u_2/u_1)^2 > 1\}$ must be a neighborhood of one or both of the boundary points. Note $x_0$ cannot be $-L/2$ or $L/2$. If $V$ is even then $u_1^2$ and $u_2^2$ are also both even and so we have that $\{u_2^2 > u_1^2\}$ is a neighborhood of both boundary points.

Since $V$ is bounded $u_j \in H^2(I)$ so that $u_j$ is twice differentiable a.e.\ and the eigenvalue equation holds for a.e.\ $x \in I$. By the fundamental theorem of calculus and the boundary condition at $x = -L/2$ we have, for $x \in I$,
\[\bigg( \frac{u_2}{u_1} \bigg)'(x) = u_1^{-2}(u_2' u_1 - u_2 u_1')(x) = u_1^{-2}\int_{-L/2}^x (u_2' u_1 - u_2 u_1')' \, dt.\]
Computing the derivative in the integrand and using the eigenvalue equation we have
\[(u_2' u_1 - u_2 u_1')(x) =  -(\lambda_2 - \lambda_1)\int_{-L/2}^{x} u_1u_2 \, dt < 0, \quad \text{for all} ~ x \in I,\]
since $u_1$ and $u_2$ are orthogonal on $I$ and $u_2$ only changes sign once there, completing the proof.
\end{proof}

Let $M \geq 0$ and $\underline{\alpha}, \overline{\alpha},\underline{\beta},\overline{\beta} \in (-\infty,\infty)$ be such that $\underline{\alpha} \leq \overline{\alpha}$ and $\underline{\beta} \leq \overline{\beta}$. Define the classes of potentials and Robin parameters
\[\Ss = \{(V,\alpha,\beta) \in L^{\infty}(I) \times [\underline{\alpha}, \overline{\alpha}] \times [\underline{\beta}, \overline{\beta}] ~:~ V \text{ single-well with } \tau = 0 \text{ and } 0 \leq V \leq M\}\]
and
\begin{equation}
\label{eq:ConvSet}
    \C = \{(V,\alpha,\beta) \in L^{\infty}(I) \times [\underline{\alpha}, \overline{\alpha}] \times [\underline{\beta}, \overline{\beta}] ~:~ V \text{ convex and } 0 \leq V \leq M\}.
\end{equation}
\begin{lemma}
\label{lem:MinEx}
The gap $\Lambda(V,(\alpha,\beta))$ has minimizers in each of the classes $\Ss$ and $\C$.
\end{lemma}

\begin{proof}
To begin we prove $\Ss$ and $\C$ are both sequentially compact for the product topology given by $L^1(I)$-convergence of $V$ and the standard topology on $[\underline{\alpha}, \overline{\alpha}] \times [\underline{\beta}, \overline{\beta}]$. Since $[\underline{\alpha}, \overline{\alpha}] \times [\underline{\beta}, \overline{\beta}]$ is compact it suffices to prove compactness for the first factor so let $\{V_n\}_n$ be a sequence in $\Ss$ or $\C$. For $\Ss$ the Helly selection theorem implies there is a subsequence of $\{V_n\}_n$ that converges pointwise to an element of $\Ss$. Since the potentials in $\Ss$ are uniformly bounded by the constant $M$ the dominated convergence theorem implies $L^1$-convergence holds. For $\C$ this follows because the potentials in $\C$ are uniformly equicontinuous on compact subsets of $I$ so the Arzela--Ascoli theorem implies that $\{V_n\}_n$ has a subsequence that converges uniformly on compact sets. This shows that if $\{(V_n,\alpha_n,\beta_n)\}_n$ in  $\Ss$ or $\C$ is a minimizing sequence for $\Lambda$ then there exists a convergent subsequence in the above topology. 

\bigskip

To show the gap has a minimizer it suffices to show the Robin gap is sequentially lower-semicontinuous under such limits. To see this let $\{(\lambda_n,u_n)\}_n$ be a sequence of eigenpairs of fixed index of the Schr\"odinger operator given by $\{(V_n,\alpha_n,\beta_n)\}_n$. Also assume that each $u_n$ has unit $L^2$-norm. We will show that $\{u_n\}_n$ is bounded in $H^1(-L/2,L/2)$ and use the Rellich--Kondrachov compactness theorem to extract a convergent subsequence. 

First we multiply the eigenvalue equation by $u_n$ and integrate by parts to find
\begin{equation}
\label{eq:Quad}
\int_I [(u_n')^2 + V_n u_n^2] \, dx + \alpha_n u_n(-L/2)^2 + \beta_n u_n(L/2)^2 = \lambda_n,
\end{equation}
since $u_n$ is $L^2$-normalized. Note that since $\lVert V_n \rVert_{L^{\infty}}$ and the Robin parameters are uniformly bounded in $n$ we know $|\lambda_n| \leq C$ for some $C$ independent of $n$ since the quadratic form on the left side of (\ref{eq:Quad}) is monotone in the potential and Robin parameters. Since $V_n \geq 0$ we have that
\begin{equation}
\label{eq:BndryVals}
\lVert u_n' \rVert_{L^2}^2 + \alpha_0 \{ u_n(-L/2)^2 + u_n(L/2)^2\} \leq C,
\end{equation}
where $\alpha_0 = \min\{\underline{\alpha},\underline{\beta}\}$. Note that when $\underline{\alpha},\underline{\beta} \geq 0$ we have that $\alpha_0 \geq 0$ and so the boundary terms can be dropped to show $\{u_n\}_n$ is bounded in $H^1(I)$. When $\alpha$ is negative a more delicate argument is necessary to show boundedness. This can be done by using the inequality
\[|u_n(x) - u_n(y)| \leq \lVert u_n' \rVert_{L^2}|x - y|^{1/2}, \quad \text{for each } x,y \in I\]
to estimate the boundary values of $u_n$ in terms of $\lVert u_n' \rVert_{L^2}$. Such an argument can be found in the book by Widom \cite[ Chapter V, Theorem 2]{Widom}.

Now we extract a subsequence such that $\lambda_n,\alpha_n,$ and $\beta_n$ converge respectively to $\lambda_{\infty} \in \R,\alpha_{\infty} \in [\underline{\alpha},\overline{\alpha}],$ and $\beta_{\infty}  \in [\underline{\beta},\overline{\beta}]$. The Rellich--Kondrachov compact embedding theorem implies there is a further subsequence such that $u_n \to u_{\infty}$ in $L^2(I)$ and $u_n \rightharpoonup u_{\infty}$ in $H^1(I)$. Note that the trace map that sends a function in $H^1(I)$ to its boundary values in $L^2(\{-L/2,L/2\})$ is compact. Equivalently, the trace map is completely continuous so that weak convergence of $u_n$ implies $u_n(\pm L/2) \to u_{\infty}(\pm L/2)$.

The weak formulation of the eigenvalue problem for each $n$ is
\[\int_I [u_n'v' + V_n u_n v] \, dx + \alpha_n u_n(-L/2)v(-L/2) + \beta_n u_n(L/2)v(L/2) = \lambda_n \int_I u_n v \, dx, \quad \text{for each } v \in H^1(I).\]
Using the above convergence and that $V_n \to V_{\infty}$ in $L^1(I)$ we can take $n \to \infty$ to find
\[\int_I [u_{\infty}'v' + V_{\infty} u_{\infty} v] \, dx + \alpha_{\infty} u_{\infty}(-L/2)v(-L/2) + \beta_{\infty} u_{\infty}(L/2)v(L/2) = \lambda_{\infty} \int_I u_{\infty} v \, dx, \quad \text{for each } v \in H^1(I).\]
Since $\lVert u_{\infty} \rVert_{L^2} = 1 \neq 0$, $u_{\infty}$ is a weak eigenfunction with eigenvalue $\lambda_{\infty}$. 

This shows that $\lambda_n$ converges to an element of the spectrum of the problem given by $(V_{\infty},\alpha_{\infty},\beta_{\infty})$. Moreover, the first eigenvalue of the sequence of problems given by $(V_n,\alpha_n,\beta_n)$ converges to the first eigenvalue of the problem given by $(V_{\infty},\alpha_{\infty},\beta_{\infty})$. This follows because ground states are characterized as the non-negative eigenfunctions. 

By extracting a final subsequence we can ensure there is a common subsequence along which the first and second eigenvalues of $(V_n,\alpha_n,\beta_n)$ converge to $\lambda_1(V_{\infty},(\alpha_{\infty},\beta_{\infty}))$ and some element of the spectrum of $(V_{\infty},(\alpha_{\infty},\beta_{\infty}))$. Additionally, since the eigenfunctions corresponding to distinct eigenvalues are orthogonal for each $n$ and each eigenspace is 1-dimensional their limits lie in distinct eigenspaces. In particular, this shows that $j \geq 2$ so that \[\Lambda(V_n,(\alpha_n,\beta_n)) \to \lambda_j(V_{\infty},(\alpha_{\infty},\beta_{\infty})) - \lambda_1(V_{\infty},(\alpha_{\infty},\beta_{\infty})) \geq \Lambda(V_{\infty},(\alpha_{\infty},\beta_{\infty})), \quad \text{ as } n \to \infty.\]
Thus, the gap is sequentially lower semi-continuous on $\Ss$ and $\C$.
\end{proof}

Now we prove a formula for the derivative of an eigenvalue as the potential and boundary parameters are varied. The main tool for the proof comes from a Banach space version of the implicit function theorem. The theory of calculus in Banach spaces and the implicit function theorem is developed by Buffoni and Toland in \cite[Part 1]{BT}.

\begin{lemma}
\label{lem:Deriv}
Assume $\{(V(\cdot,t),\alpha(t),\beta(t))\}_{t \in \R}$ is a family of bounded potentials such that $\frac{\partial^k V}{\partial t^k}(\cdot,t) \in L^{\infty}(I)$ and $\alpha$ and $\beta$ are $k$-times continuously differentiable in $t$, $k \geq 1$. If $\{(\lambda(t),u(\cdot,t))\}_{t \in \R}$ is a family of eigenpairs of a fixed index where $u(\cdot,t)$ is $L^2$-normalized then $\lambda(t) = \lambda(V(\cdot,t),\alpha(t),\beta(t))$ is $k$-times differentiable and
\[\frac{d \lambda}{dt}(t) = \int_I \frac{\partial V}{\partial t} u^2 \, dx +  \alpha'(t) u(-L/2)^2 + \beta'(t) u(L/2)^2.\]
\end{lemma}

Let $X$ and $Y$ be Hilbert spaces and $(\mathcal{L}(X,Y),\lVert \cdot \rVert)$ denote the Banach space of bounded linear operators between $X$ and $Y$ equipped with the operator norm $\lVert \cdot \rVert$.

\begin{proof}

To prove that the eigenpairs are smooth in $t$ we transform the eigenvalue problem into one with Neumann boundary conditions so that the problem can be defined on a fixed Hilbert space $X$. Since each eigenvalue is simple we will then be able to apply Proposition 3.6.1 in \cite{BT} to conclude that each curve of eigenvalues and eigenfunctions of given index is smooth.

First let $\{\chi_1,\chi_2\}$ be a partition of unity subordinate to the cover $\{[-L/2,L/4),(-L/4,L/2]\}$ and define $\rho(x,t) = \chi_1(x) e^{-\alpha(t)x} + \chi_2(x) e^{\beta(t)x}$. It follows that $\rho,\partial_x \rho \in C^k([-L/2,L/2] \times \R)$ and $\rho$ is uniformly positive on $[-L/2,L/2] \times [a,b]$ for each $a<b$. A computation shows $w = \rho u$ satisfies
\begin{equation*}
%\label{eq:RobinProb}
\begin{cases}
      -w'' + (\log(\rho^2))'w' + \widetilde V w = \lambda w, \quad \text{on} ~ (-L/2,L/2) \\
      w'(\pm L/2) = 0,
\end{cases}
\end{equation*} 
where $\widetilde V = \frac{\rho''}{\rho} - 2(\frac{\rho'}{\rho})^2 + V$. After multiplying and dividing by the integrating factor $\mu(x,t) = \rho(x,t)^{-2}$ the first two terms can be combined to give a self-adjoint problem
\begin{equation*}
%\label{eq:RobinProb}
\begin{cases}
      -\frac{1}{\mu}(\mu w')' + \widetilde V w = \lambda w, \quad \text{on} ~ (-L/2,L/2) \\
      w'(\pm L/2) = 0,
\end{cases}
\end{equation*} 
Now we check that the operator valued function $L : \R \to \mathcal{L}(X,L^2(I))$ defined by
\[L(t)w = -\frac{1}{\mu}(\mu w')' + \widetilde Vw,\]
is $k$-times Fr\'echet differentiable, where $X$ is the subspace of $H^2(I)$ consisting of functions with zero derivative at the boundary points $x = \pm L/2$. Let $A \in \mathcal{L}(X,L^2(I))$ be defined by
\[Aw = \frac{-1}{\mu}(\dot{\mu}w')' + \frac{\dot{\mu}}{\mu^2}(\mu w')' + \dot{\widetilde V}w\] 
and
\[(D_hf)(x,t) = \frac{f(x,t + h) - f(x,t)}{h}, \quad \text{for } h > 0,\]
be the difference quotient operator in $t$. Fr\'echet differentiability follows from showing $\lVert D_h L - A \rVert \to 0$ as $h \to 0$, where $\lVert \cdot \rVert$ is the operator norm on $\mathcal{L}(H^2, L^2)$.
This limit follows from a dominated convergence argument using the regularity of $\rho$ and $V$. Showing higher order Fr\'echet differentiability is similar.

Since $\lambda(t)$ is a simple eigenvalue Proposition 3.6.1 in \cite{BT} shows that there is a $C^k$-curve ($k \geq 1$) of simple eigenpairs near each $t$. It follows from the second part of Proposition 3.6.1 that there is a single smooth curve of eigenpairs for all $t \in \R$. Finally, we can renormalize the eigenfunction to have unit $L^2$-norm as necessary.

\bigskip

Now we prove the formula for the derivative of the eigenvalue. Let $\dot{f} = \frac{\partial f}{\partial t}$ denote the derivative of $f$ with respect to $t$. Recall the weak formulation of the Schr\"odinger eigenvalue problem is
\begin{equation}
\label{eq:WF}
\int_I [u' v' + W u v] \, dx = \lambda \int_I u v \, dx, \quad \text{for each } v \in H^1(I),
\end{equation}
where $W = V + \alpha(t)\delta_{-L/2} + \beta(t) \delta_{L/2}$ and $\delta_p$ denotes the Dirac delta function at $p \in [-L/2,L/2]$. We also have $\dot{W} = \dot{V} + \dot{\alpha}(t)\delta_{-L/2} + \dot{\beta}(t) \delta_{L/2}$ and observe that $\dot{u} \in H^2(I)$ so the dominated convergence theorem implies the $t$-derivative of the weak formulation is
\begin{equation}
\label{eq:DWF}
\int_I [\dot{u}' v' + \dot{W} u v + W \dot{u} v] \, dx = \dot{\lambda} \int_I u v \, dx + \lambda \int_I \dot{u} v \, dx, \quad \text{for each } v \in H^1(I).
\end{equation}

Since $u$ is $L^2$-normalized we have that $\int\dot{u} u \, dx = 0$ for each $t$. Letting $v = u$ in (\ref{eq:DWF}) and $v = \dot{u}$ in (\ref{eq:WF}) we have
\[\int_I [\dot{u}' u' + \dot{W} u^2 + W \dot{u} u] \, dx = \dot{\lambda} \quad \text{and}  \quad \int_I [u' \dot{u}' + Wu \dot{u}] \, dx = 0.\]
Combining these last two results proves the formula $\dot{\lambda} = \int_I \dot{W} u^2 \, dx$, and hence the forumla displayed in the lemma.

\end{proof}

\begin{lemma}
\label{lem:L1Concave}
If $V_0$ is non-negative and positive on a set of positive measure then $t \mapsto \lambda_1(tV_0,(\alpha,\beta))$ is strictly increasing and concave.
\end{lemma}

\begin{proof}
Let $u_t$ be an $L^2$-normalized ground state corresponding to $\lambda_1(t) = \lambda_1(tV_0,(\alpha,\beta))$. The function $t \mapsto \lambda_1(t)$ is smooth and strictly increasing by Lemma \ref{lem:Deriv} because 
\[\lambda_1'(t) = \int_I V_0 u_t^2 \, dx > 0,\] 
since $u_t^2$ is positive on $I$.

To see the concavity statement recall the variational characterization \[\lambda_1(t) = \inf R_t(u),\]
where $R_t : H^1(I) \setminus \{0\} \to \R$ is the Rayleigh quotient
\[R_t(u) = \frac{\int_I (u')^2 + t V_0 u^2 \, dx + \alpha u(-L/2)^2 + \beta u(L/2)^2}{\int_I u^2 \, dx}.\]
Hence the function $t \mapsto \lambda_1(t)$ is concave because it is the infimum of the family of linear functions $\{t \mapsto R_t(u)\}_u$ (see, for example, \cite[p.\ 153-154, item (vi) of Gloss (3.5.24)]{Thirring}).
\end{proof}

\begin{remark}
In the above proof we show $t \mapsto \lambda_1(tV_0,(\alpha,\beta))$ is a concave function, but in fact it is a strictly concave function under the hypotheses in Lemma \ref{lem:L1Concave}. This follows from the second-order perturbation formula
\[\frac{d^2}{dt^2} \lambda_1(tV_0,(\alpha,\beta)) = -2\sum_{j \geq 2} (\lambda_j - \lambda_1)^{-1} |\langle u_1,V_0 u_j \rangle|^2 < 0,\]
see \cite[p.\ 7, \S XII.1, and p.\ 21]{RS} for example.
\end{remark}

%%%%%%%%%%%%%%%%%%%%%%%%%%%%%%%%%%%%%%%%%%%%%%%%%%%%%%%%%%%%%%%%%%%%%%%%%%%%%%%%%%%%%%%%%%%%%%%%%%%%%%%%%%%%%%%%%%%%%%%%%%%%%%%%%%%%%%%%%%

\section{\textbf{Single-well potentials with centered transition point}}

In this section we assume that $\beta = \alpha$. Since the Robin eigenvalues satisfy the scaling relation $\lambda_j(t^{-2}V(\cdot/t),\alpha/t;tI) = t^{-2}\lambda_j(V,\alpha;I)$ we can work on the interval $(-\pi/2,\pi/2)$ of length $L = \pi$. The bulk of the proof concerns the roots of a trigonometric equation that gives the eigenvalues of the Schr\"odinger operator
\[H_m = -\frac{d^2}{dx^2} + m \mathbf{1}_{(0,\pi/2)}, \quad \text{for } m \geq 0,\] 
acting on $L^2(-\pi/2,\pi/2)$ with Robin boundary conditions. In this argument we follow Horv\'ath \cite{Horvath}, making the generalizations necessary to apply his method to the Robin problem for $\alpha \geq 0$ (Horv\'ath dealt with the Dirichlet case $\alpha = \infty$). 

We denote the eigenvalues and gap for $V = 0$ by $\lambda_j(\alpha) = \lambda_j(0,\alpha)$ and $\Lambda(\alpha) = \Lambda(0,\alpha)$ for notational convenience. We will show that the eigenvalues of this operator are the real solutions of the transcendental equation 
\begin{equation}
\label{eq:trans}
    f_{\alpha}(t) = - f_{\alpha}(t - m).
\end{equation}
Here $f_{\alpha} : \R \to \widetilde \R$ is defined by
\[f_{\alpha}(t) = \sqrt{t} \, \frac{\alpha \cos(\sqrt{t} \frac{\pi}{2}) - \sqrt{t}\sin(\sqrt{t} \frac{\pi}{2})}{\sqrt{t}\cos(\sqrt{t} \frac{\pi}{2}) + \alpha \sin(\sqrt{t} \frac{\pi}{2})}, \quad \text{for } t \neq 0,\]
and $2\alpha/(\alpha \pi + 2)$ for $t = 0$, where $\widetilde \R$ is the extended real numbers $[-\infty,+\infty]$ with $-\infty$ identified with $+\infty$. In particular, $f_{\alpha}(t) \in \R$ when $t$ is negative. Note that in the limit $\alpha \to \infty$ we recover the function $f(t) = \sqrt{t} \cot(\sqrt{t} \frac{\pi}{2})$ in \cite[\S 2]{Horvath}. Let $t_j = t_j(m)$ for $j = 1,2,3, \dots$ denote the real solutions of (\ref{eq:trans}) listed in increasing order.

\begin{lemma}
\label{lem:Trig}
\hspace{0mm}\\
(i) If $\alpha \in \R$ then $t_j(m)$ is the $j^{\text{th}}$ eigenvalue of $H_m$ with Robin boundary conditions.\\ 
(ii) If $\alpha \in [0,\infty)$ and $m \geq 0$ then
\[t_2 - t_1 \geq \Lambda(0,\alpha)\]
with equality if and only if $m = 0$.
\end{lemma}

\begin{proof}
\textbf{Part (i)}: First assume that $m > 0$. Let $y_1(x;\lambda) = \cos(\sqrt{\lambda}x)$ and $y_2(x;\lambda) = \lambda^{-1/2} \sin(\sqrt{\lambda}x)$ for $\lambda \neq 0$, with $y_2(x;0)$ defined as a limit. The functions $y_1$ and $y_2$ form a fundamental set of solutions to $y'' + \lambda y = 0$ defined and analytic in $\lambda$ for all $\lambda \in \mathbb{C}$ and real-valued for all $\lambda \in \R$. In terms of $y_1$ and $y_2$, $f_{\alpha}$ becomes
\[f_{\alpha}(t) = \frac{y_1'(\pi/2;t) + \alpha y_2'(\pi/2;t)}{y_1(\pi/2;t) + \alpha y_2(\pi/2;t)}\]
and as such is meromorphic in $\lambda$. Thus in these terms and in general the singularity in $f_\alpha$ at $t = 0$ is only apparent, and can always be handled via a limiting argument (as done above).

An eigenfunction $u$ with eigenvalue $\lambda$ of $H_m$ is of the form
\[u(x) = \begin{cases}
A y_1(x;\lambda) + B y_2(x;\lambda), \quad &x \in (-\pi/2,0)\\
C y_1(x;\lambda-m) + D y_2(x;\lambda-m), \quad &x \in (0,\pi/2),
\end{cases}\]
for coefficients $A,B,C,D \in \R$ (not all zero). We view the sine and cosine as functions on the complex plane so that when $\lambda$ or $\lambda - m$ is negative we instead have a linear combination of hyperbolic sines and cosines.

Note that $u \in C^1(-\pi/2,\pi/2)$ since $u \in H^2_{loc}(-\pi/2,\pi/2)$ by the weak eigenequation and definition of the weak derivative. From the smoothness of $u$ and the Robin boundary conditions we will deduce that for fixed $\alpha \in \R$ the eigenvalues $\lambda$ are solutions of the transcendental equation (\ref{eq:trans}). The fact that $u \in C^1(-\pi/2,\pi/2)$ implies that the left and right limits of $u$ and $u'$ are equal at zero. It follows that
\begin{equation}
\label{eq:ABCD}
     A = C \quad \text{and} \quad B = D.
\end{equation}
The Robin boundary condition at $x = -\pi/2$ implies that
\begin{equation}
\label{eq:BCL}
\frac{B}{A} = \frac{- y_1'(-\pi/2;\lambda) + \alpha y_1(-\pi/2;\lambda)}{y_2'(-\pi/2;\lambda) - \alpha y_2(-\pi/2;\lambda)} = \sqrt{\lambda} \frac{\alpha \cos(\sqrt{\lambda} \frac{\pi}{2}) - \sqrt{\lambda}\sin(\sqrt{\lambda} \frac{\pi}{2})}{\sqrt{\lambda}\cos(\sqrt{\lambda} \frac{\pi}{2}) + \alpha \sin(\sqrt{\lambda} \frac{\pi}{2})},
\end{equation}
and the boundary condition at $x = \pi/2$ implies that
\begin{equation}
\label{eq:BCR}
\frac{D}{C} = -\frac{ y_1'(\pi/2;\lambda^*) + \alpha y_1(\pi/2;\lambda^*)}{y_2'(\pi/2;\lambda^*) + \alpha y_2(\pi/2;\lambda^*)} = -\sqrt{\lambda^*}\frac{\alpha \cos(\sqrt{\lambda^*} \frac{\pi}{2}) - \sqrt{\lambda^*}\sin(\sqrt{\lambda^*} \frac{\pi}{2})}{\sqrt{\lambda^*}\cos(\sqrt{\lambda^*} \frac{\pi}{2}) + \alpha \sin(\sqrt{\lambda^*} \frac{\pi}{2})},
\end{equation}
where $\lambda^* = \lambda - m$.

The relations in (\ref{eq:ABCD}) imply that $B /A= D/C$ so that we can equate the right sides of (\ref{eq:BCL}) and (\ref{eq:BCR}) to obtain the transcendental equation (\ref{eq:trans}) for $t \neq 0,m$ (with the cases $t = 0,m$ following as limits). 

%To deduce that the transcendental equation holds when $t = 0$, note that the middle expression in (\ref{eq:BCL}) is $2 \alpha/(\alpha \pi + 2)$ when $\lambda = 0$.

This calculation shows that each eigenvalue of $H_m$ is a solution of (\ref{eq:trans}) for $m > 0$. Conversely, each solution of (\ref{eq:trans}) gives a $\lambda$ such that $u$ is an eigenfunction of $H$ with Robin boundary conditions. It follows that the Robin eigenvalues of $H_m$ are in 1-1 correspondence with the solutions of (\ref{eq:trans}).

Finally, assume that $m = 0$. Observe that $f_{\alpha}(t)$ is finite except when $-\sqrt{t} \cot(\sqrt{t}\pi/2) = \alpha$. The solutions of this equation are the eigenvalues with associated eigenfunction that has odd symmetry, or equivalently, the eigenvalues $\lambda_{2j}(\alpha)$ for $j \in \N$. Similarly, the zeros of the numerator of $f_{\alpha}$ are the eigenvalues $\lambda_{2j - 1}(\alpha)$ for $j \in \N$. It follows that the solutions of $f_{\alpha}(t) = -f_{\alpha}(t)$ are the eigenvalues of $H_m$ when $m = 0$.

\bigskip

\textbf{Part (ii)}: We show the spectral gap $\Lambda(m \mathbf{1}_{(0,\pi/2)},\alpha) = (t_2 - t_1)(m)$ is a strictly increasing function of $m$ for $m \in (0,m_0]$ and that on $(m_0,\infty)$ the spectral gap is strictly larger than $\Lambda(\alpha)$, where $m_0 = t_2^{-1}(\lambda_3(\alpha)) > 0$.

To see that $m_0$ is well-defined first note that each $t_j$ is strictly increasing in $m$ since the potential is strictly increasing pointwise in $m$ for $x \in (0,\pi/2)$. Additionally, $t_2(m) \nearrow \lambda_{4}(\alpha)$ as $m \to \infty$ since $\lambda_2(\alpha) \leq t_2(m) < \lambda_4(\alpha)$ is continuous and when $t = t_2(m)$ the right side of (\ref{eq:trans}) tends to $- \infty$ as $m \to \infty$ as $t_2(m)$ is uniformly bounded. Alternatively, one can view $m \to \infty$ as the ``infinite potential-well" limit so that $\lambda_j(m \mathbf{1}_{(0,\pi/2)},\alpha;(-\pi/2,\pi/2)) \to \lambda_j(0,(\alpha,\infty);(-\pi/2,0)) = \lambda_{2j}(0,\alpha;(-\pi/2,\pi/2))$ as $m \to \infty$. Hence, $m \mapsto t_j(m)$ is an invertible function on $[0,\infty)$ and $m_0 = t_2^{-1}(\lambda_3(\alpha))$ exists in $(0,\infty)$. 

In fact, $m_0$ is given explicitly by $m_0 = \lambda_3(\alpha) - \lambda_1(\alpha)$. This follows because the choices $t = \lambda_3(\alpha), m = \lambda_3(\alpha) - \lambda_1(\alpha),$ give $f_\alpha(t) = f_\alpha(\lambda_3(\alpha)) = 0$ and also $f_\alpha(t - m) = f_\alpha(\lambda_1(\alpha)) = 0$.

\bigskip

\noindent \underline{Reduction to $m \in (0,m_0)$}: First we deduce that $t_2 - t_1 > \Lambda(\alpha)$ when $m \geq m_0$. In this case, $t_2(m) \geq \lambda_3(\alpha)$ and $t_1(m) < \lambda_1(0,(\alpha,\infty);(-\pi/2,0))$ so that
\[t_2(m) - t_1(m) > \lambda_3(\alpha) - \lambda_1(0,(\alpha,\infty);(-\pi/2,0)) = \lambda_3(\alpha) - \lambda_2(\alpha) > \lambda_2(\alpha) - \lambda_1(\alpha) = \Lambda(\alpha).\]
The final inequality follows from the monotonicity of the gaps with respect to the Robin parameter \cite[\S 4]{Smits}: $\lambda_3(\alpha) - \lambda_2(\alpha) \geq \lambda_3(0) - \lambda_2(0) = 3$ and $\lambda_2(\alpha) - \lambda_1(\alpha) < \lambda_2(\infty) - \lambda_1(\infty) = 3$.

\bigskip

\noindent \underline{$t_2 - t_1$ strictly increasing on $[0,m_0]$}: We begin by proving that $f'_{\alpha} < 0$ on each connected component of $f_{\alpha}^{-1}((-\infty,\infty)) = \R \setminus \{\lambda_{2j}(\alpha) : j \in \N\}$. It is useful to view
\[Q(t) = \frac{-S(t)}{C(t)} = \frac{\alpha \cos(\sqrt{t} \frac{\pi}{2}) - \sqrt{t}\sin(\sqrt{t} \frac{\pi}{2})}{\sqrt{t}\cos(\sqrt{t} \frac{\pi}{2}) + \alpha \sin(\sqrt{t} \frac{\pi}{2})},\]
as a generalized tangent function (consider $\alpha = 0$). Note that $f_{\alpha}(t) = \sqrt{t}Q(t)$ so that
\[f_{\alpha}'(t) = \frac{1}{2 \sqrt{t}} Q(t) + \sqrt{t} Q'(t).\]

A direct calculation shows that
\begin{equation}
\label{eq:DerivQ}
Q'(t) = \frac{-S'(t)C(t) + S(t)C'(t)}{C(t)^2} = -\frac{1}{4\sqrt{t}}\frac{\alpha(\alpha \pi + 2) + \pi t}{C(t)^2}
\end{equation}
so that
\[f'_{\alpha}(t) = \frac{-1}{2 \sqrt{t}} \frac{S(t)}{C(t)} - \frac{1}{4}\frac{\alpha(\alpha \pi + 2) + \pi t}{C(t)^2} = \frac{-2S(t)C(t) - \alpha(\alpha \pi + 2) \sqrt{t} - \pi t \sqrt{t}}{4 \sqrt{t}C(t)^2}.\]

Using the ``double-angle" trigonometric identities we have
\begin{align*}
    -2S(t)C(t) =& 2 \alpha \sqrt{t} \bigg(\cos^2\Big(\sqrt{t}\frac{\pi}{2} \Big) - \sin^2 \Big(\sqrt{t} \frac{\pi}{2} \Big) \bigg) - 2(t - \alpha^2)\sin\Big(\sqrt{t} \frac{\pi}{2} \Big)\cos\Big(\sqrt{t} \frac{\pi}{2}\Big)\\
    =& 2 \alpha \sqrt{t} \cos(\sqrt{t}\pi) - (t - \alpha^2)\sin(\sqrt{t}\pi).
\end{align*}
Thus, 
\begin{align}
\label{eq:fDeriv}
    f'_{\alpha}(t) &= \frac{(2\alpha\cos(\sqrt{t}\pi) - \alpha(\alpha \pi + 2))\sqrt{t} - ((t - \alpha^2)\sin(\sqrt{t} \pi) + \pi t \sqrt{t})}{4 \sqrt{t} C(t)^2}\nonumber \\
    &= -\frac{2 \alpha \sqrt{t}(1 - \cos(\sqrt{t} \pi)) + \alpha^2(\sqrt{t}\pi - \sin(\sqrt{t}\pi)) + t(\sqrt{t} \pi + \sin(\sqrt{t} \pi))}{4 \sqrt{t} C(t)^2}.
\end{align}
When $t > 0$ using that $\cos(\cdot) \leq 1$ and that $|\sin(\sqrt{t} \pi)| < \sqrt{t} \pi$ in (\ref{eq:fDeriv}) we have (recall that $\alpha \geq 0$)
\[f'_{\alpha}(t) \leq \frac{-t(\sqrt{t} \pi + \sin(\sqrt{t} \pi))}{4 \sqrt{t} C(t)^2} \leq \frac{-t(\sqrt{t} \pi - |\sin(\sqrt{t} \pi)|)}{4 \sqrt{t} C(t)^2} < 0.\]

Formula (\ref{eq:fDeriv}) actually holds for all $t \in \mathbb{C}$. This means we can take $t$ to be a negative real number. Hence, we use the identities $\cos(i \theta) = \cosh(\theta)$ and $\sin(i \theta) = i\sinh(\theta)$ to see that
\begin{equation*}
%\label{eq:fDerivNeg}
f'_{\alpha}(t) = -\frac{2 \alpha \sqrt{-t}(\cosh(\sqrt{-t} \pi) -1) + \alpha^2(\sinh(\sqrt{-t} \pi) - \sqrt{-t}\pi) + (-t)(\sqrt{-t} \pi + \sinh(\sqrt{-t} \pi))}{4 \sqrt{-t} (\sqrt{-t}\cosh(\sqrt{-t} \frac{\pi}{2}) + \alpha \sinh(\sqrt{-t} \frac{\pi}{2}))^2}
\end{equation*}
holds for $t < 0$. To see that $f_{\alpha}'(t) < 0$ for $t<0$, notice that $\cosh(\cdot) > 1$ and $\sinh(\sqrt{-t}\pi) > \sqrt{-t} \pi$ so that 
\[f'_{\alpha}(t) < -\frac{(-t)(\sqrt{-t} \pi + \sinh(\sqrt{-t} \pi)) )}{4 \sqrt{-t} (\sqrt{-t}\cosh(\sqrt{-t} \frac{\pi}{2}) + \alpha \sinh(\sqrt{-t} \frac{\pi}{2}))^2} < 0, \quad \text{for } t < 0.\]
And finally, for $t = 0$ one can take the limit to see that $f'_\alpha(0) = -\pi(\pi^2 \alpha^2 + 6 \pi \alpha + 12)/[6 (\pi \alpha + 2)^2] < 0$.

Monotonicity of the gap in $m$ will follow from showing 
\begin{equation}
\label{eq:dti}
\frac{dt_1}{dm} \leq \frac{1}{2} < \frac{dt_2}{dm}, \quad \text{on } (0,m_0).
\end{equation}
To prove the left inequality of (\ref{eq:dti}) (which actually holds for all $m > 0$) observe that we can compute that $t_1'(0) = 1/2$ (via perturbation theory) since  when $m = 0$ the ground state is symmetric and $L^2$-normalized. The inequality follows because Lemma \ref{lem:L1Concave} shows that $m \mapsto t_1(m)$ is a concave function.

The right inequality in (\ref{eq:dti}) is more delicate. Take $t = t_2(m)$ in (\ref{eq:trans}) and differentiate implicitly with respect to $m$ to find
\[\frac{1 - t_2'(m)}{t_2'(m)} = \frac{f_{\alpha}'(t_2(m))}{f_{\alpha}'(t_2(m) - m)}.\]
From this formula and that $f_{\alpha}' < 0$ it follows that the right inequality of (\ref{eq:dti}) is equivalent to
\begin{equation*}
%\label{eq:DerivIneq} 
f_{\alpha}'(t_2 - m) < f_{\alpha}'(t_2),
\end{equation*}
which we prove now.

It follows from (\ref{eq:DerivQ}) that
\begin{equation}
\label{eq:Derivf}
    f'_{\alpha}(t) = \frac{1}{2 \sqrt{t}} Q(t) - \frac{\alpha(\alpha \pi + 2) + \pi t}{4(t + \alpha^2)} \big( 1 + Q(t)^2 \big)
\end{equation}
since $C(t)^2 + S(t)^2 = t + \alpha^2$ so that $C(t)^{-2} = (t + \alpha^2)^{-1} \big( 1+ (S(t)/C(t))^2 \big)$. The transcendental equation $f_{\alpha}(t_2) = -f_{\alpha}(t_2 - m)$ implies that 
\[Q(t_2 - m) = -\frac{\sqrt{t_2}}{\sqrt{t_2 - m}} Q(t_2).\]
Using this and (\ref{eq:Derivf}) we can compute
\[f_{\alpha}'(t_2 - m) =  -\frac{\sqrt{t_2}}{2(t_2 - m)}Q(t_2) - h(t_2-m) \Big( 1 + \frac{t_2}{t_2 - m} Q(t_2)^2 \Big),\]
where 
\[h(t) = \frac{\alpha(\alpha \pi + 2) + \pi t}{4(t + \alpha^2)} = \frac{\pi}{4} + \frac{\alpha}{2(t + \alpha^2)}\]
and
\begin{equation}
\label{eq:ft_2}
f_{\alpha}'(t_2 - m) - f_{\alpha}'(t_2) 
= \frac{m-2t_2}{2(t_2 - m)\sqrt{t_2}}Q(t_2) - \big(h(t_2 - m) - h(t_2)\big) - \bigg(\frac{t_2}{t_2 - m}h(t_2 - m) - h(t_2) \bigg)Q(t_2)^2.
\end{equation}

The roots of $S(t)$ and $C(t)$ are, respectively, the eigenvalues $\lambda_{2j-1}(\alpha)$ and $\lambda_{2j}(\alpha)$ for $j \in \N$ so that $f_{\alpha}(t)$ and $Q(t)$ are positive for $t \in (\lambda_2(\alpha),\lambda_3(\alpha))$ since they are decreasing. In particular, this shows that $f_{\alpha}(t_2) > 0$ when $m \in (0,m_0)$.

We know that $t_2 > \lambda_1(\alpha) + m$ since otherwise $t_2 \leq \lambda_1(\alpha) + m$ so that $-f_{\alpha}(t_2 - m) \leq 0$ but $-f_{\alpha}(t_2 - m) = f_{\alpha}(t_2) > 0$. This shows $t_2 -m > 0$ and $m - 2t_2 < 0$ so that the first term is negative since $Q(t_2) > 0$. The second term is less than or equal to zero since $h' \leq 0$ on $(-\alpha^2,\infty)$ when $\alpha \geq 0$. Finally, the third term is negative since $h > 0, h' \leq 0,$ and $t_2 - m > 0$. 
\end{proof}

\subsection*{Proof of Theorem \ref{thm:Horvath}}
The case $\alpha = \infty$ was proved by Horv\'ath in \cite{Horvath} so assume $\alpha \in [0,\infty)$. Since the eigenvalues are invariant under translating the interval and satisfy the scaling relation $\lambda_j(t^{-2}V(\cdot/t),\alpha/t;tI) = t^{-2}\lambda_j(V,\alpha;I)$ it suffices to carry out this calculation on the interval $(-\pi/2,\pi/2)$ of length $L = \pi$.  

We first reduce from general $V$ to step function potentials of the form
\begin{equation}
\label{eq:StepPot}
    m \mathbf{1}_{(0,\pi/2)}, \quad \text{for } m \geq 0.
\end{equation}
It follows from taking $\underline{\alpha} = \overline{\alpha} = \alpha = \underline{\beta} = \overline{\beta}$ in Lemma \ref{lem:MinEx} that there exists a minimizer $V_0$ of $\Lambda(V,\alpha)$ in the class of potentials
\[\mathcal{S}_M = \{V \in L^{\infty}(-\pi/2,\pi/2) : 0 \leq V \leq M,~ V \text{ single-well potential with transition point } \tau=0\}.\]
To see that a minimizing potential is of the form (\ref{eq:StepPot}) we consider two cases and proceed by a perturbation argument in each case. Let $x_-$ and $x_+$ be as in Lemma \ref{lem:Wrskn} for the potential $V_0$, which we assume to be a minimizing potential.

\bigskip

\noindent \textbf{Case 1: ($x_- \leq 0 < x_+$):} Define the piecewise constant potential
\begin{equation*}
V_1(x) = 
    \begin{cases}
        V_0(x_-), \quad &\text{for } x \in (-\pi/2,0],\\
        V_0(x_+), \quad &\text{for } x\in (0,\pi/2],
    \end{cases}
\end{equation*}
and consider the family of potentials
\[V(x,t) = (1 - t)V_0(x) + t V_1(x), \quad \text{for each } t \in [0,1].\]
Since $V(\cdot, t) \in \mathcal{S}_M$ for each $t$, minimality of $V_0$ implies that
\begin{equation}
\label{eq:GapDeriv}
    \frac{d \Lambda}{d t}\bigg|_{t = 0} = \int_{-\pi/2}^{\pi/2} (-V_0 + V_1) (u_2^2 - u_1^2) \, dx \geq 0.
\end{equation}
It follows from the definition of $x_-$ and $x_+$ that the integrand is pointwise non-positive. This implies that $V_0 = V_1$, except possibly at the transition point $x = 0$. When $x_+ = 0$ a similar argument works.

\bigskip

\noindent \textbf{Case 2: ($0 < x_-$):} We show that this case is impossible. First we reduce to the step function potential
\begin{equation*}
V_1(x) = 
    \begin{cases}
        V_0(0) , \quad &\text{for } x \in (-\pi/2,x_-],\\
        V_0(x_+), \quad &\text{for } x \in (x_-,\pi/2),
    \end{cases}
\end{equation*}
via a perturbation argument. Consider the family of potentials
\[V(x,t) = (1 - t)V_0(x) + t V_1(x), \quad \text{for } t \in [0,1].\]

Since $V(\cdot,t) \in \mathcal{S}_M$ for each $t$, minimality of $V_0$ implies that (\ref{eq:GapDeriv}) holds again. It follows from the definitions of $x_-$ and $x_+$ that the integrand is non-positive pointwise so that $V_0 = V_1$ everywhere, except possibly at the transition point $x = 0$. Since $V_0 - \inf_I \{V_0\} \in \Ss_M $, without loss of generality we can take $V_0 = m \mathbf{1}_{(x_-,\pi/2)}$ for some $m \geq 0$. When $x_+ < 0$ a similar argument works by symmetry.

We use the notation $\lambda_j(V,(\alpha,\beta);J)$ to denote the eigenvalues of the Schr\"odinger operator with potential $V$ and Robin boundary conditions $\alpha$ at the left endpoint and $\beta$ at the right on an interval $J$. Using that $x_- > 0$ and $x_0 \in (x_-,\pi/2)$ is the unique zero of $u_2$ we have that
\begin{align*}
\lambda_2(m \mathbf{1}_{(x_-,\pi/2)},\alpha;(-\pi/2,\pi/2)) = &\lambda_1\big(m \mathbf{1}_{(x_-,x_0)},(\alpha,\infty);(-\pi/2,x_0)\big)\\
< &\lambda_1\big(0,(\alpha,\infty);(-\pi/2,0)\big) = \lambda_2\big(0,\alpha;(-\pi/2,\pi/2)\big).
\end{align*}
The strict inequality follows from a trial function argument and that the corresponding eigenfunctions cannot coincide on $(-\pi/2,0)$. This is a contradiction since $\lambda_2(m \mathbf{1}_{(x_-,\pi/2)},\alpha) > \lambda_2(0,\alpha)$ due to monotonicity of the eigenvalues with respect to $m$. 

\bigskip

This contradiction shows that only Case 1 can occur so the the optimal potential must be a step function of the form (\ref{eq:StepPot}) (shifting $V_0$ again as necessary). Lemma \ref{lem:Trig} shows that among potentials in $\mathcal{S}_M$ the minimum is attained if and only if $m = 0$. To conclude, let $V$ be an arbitrary bounded single-well potential with transition point $\tau = 0$. Since $V - c \in \mathcal{S}_M$ for $c = \inf_I\{V\}$ and $M$ sufficiently large both the inequality and the characterization of equality hold for each such $V$.

\bigskip

Now we show the inequality $\Lambda(V,\alpha) \geq \Lambda(0,\alpha)$ can fail to hold if $V$ is a single-well potential but has a transition point $\tau \in [-L/2,L/2]$ that is not $0$ by a perturbation argument for $\alpha \in \R$. Let $u_1$ and $u_2$ denote the $L^2$-normalized eigenfunctions of the Schr\"odinger operator with $V = 0$. Without loss of generality, we can assume that $\tau \in [-L/2,0)$ by symmetry. When $\tau \in (-L/2,0)$ consider the family of potentials $V(x,t) = t \mathbf{1}_{(\tau,L/2)}(x) 
$ and note that
\[\frac{d}{d t} \Lambda(V(x,t),\alpha) = \int_{-L/2}^{L/2} \frac{dV}{dt} (u_2^2 - u_1^2) \, dx = \int_{\tau}^{L/2} (u_2^2 - u_1^2) \, dx = \int_{\tau}^0 (u_2^2 - u_1^2) \, dx < 0.\] The equality on the right follows from $u_j^2$ being even and $L^2$-normalized and the inequality follows from Lemma \ref{lem:Wrskn}. When $\tau = -L/2$ consider $V(x,t) = t \mathbf{1}_{(x_-,L/2)}$, where $x_- \in (-L/2,0)$ and $u_2^2 < u_1^2$ on $(x_-,0]$, as given by the symmetric case of Lemma \ref{lem:Wrskn}. Similarly, we have  
\[\frac{d}{d t} \Lambda(V(x,t),\alpha) = \int_{x_-}^{L/2} (u_2^2 - u_1^2) \, dx  = \int_{x_-}^0 (u_2^2 - u_1^2) \, dx < 0. \qed\]

\subsection{Gap numerics}
\label{sec:Num}

\begin{figure}
    \begin{minipage}{.48\textwidth}
    \center
    \includegraphics[scale = .55]{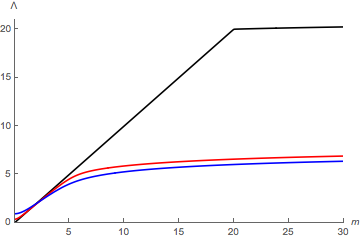}
    \end{minipage}
    \begin{minipage}{.48\textwidth}
    \center
    \includegraphics[scale = .55]{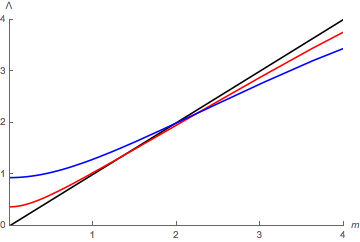}
    \end{minipage}
    \caption{The gap $\Lambda(m \boldsymbol{1}_{(0,\pi/2)},\alpha)$ plotted as a function of $m$ on the intervals $[0,30]$ and $[0,4]$ for $\alpha = -2,-1,-0.1$. The values $\Lambda(0,\alpha)$ increase with $\alpha$. Note the the curves cross because $\alpha \mapsto \Lambda(m \boldsymbol{1}_{(0,\pi/2)},\alpha)$ is not monotone for $\alpha < 0$.
}
    \label{fig:NegGapM}
\end{figure}

\begin{figure}[t]
\begin{minipage}[t]{0.48\textwidth}
    \centering
    \includegraphics[scale = .55]{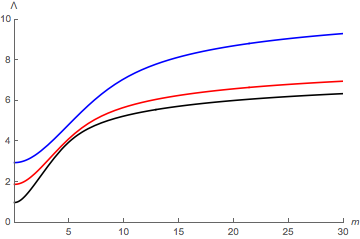}
    \caption{The gap $\Lambda(m \boldsymbol{1}_{(0,\pi/2)},\alpha)$ plotted as a function of $m$ for $\alpha = 0,2,100$. The values $\Lambda(0,\alpha)$ increase with $\alpha$.}
    \label{fig:PosGapM}
\end{minipage}
~
\begin{minipage}[t]{0.48\textwidth}
    \centering
    \includegraphics[scale = .5]{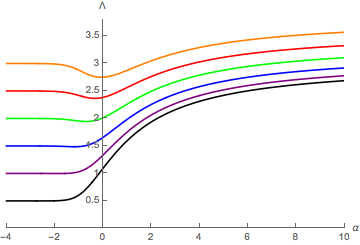}
    \caption{The gap $\Lambda(m \boldsymbol{1}_{(0,\pi/2)},\alpha)$ plotted as a function of $\alpha$ with $m = k/2$ for $1 \leq k \leq 6$. The values $\Lambda(m \boldsymbol{1}_{(0,\pi/2)},0)$ increase with $m$.}
    \label{fig:GapAlphaPlot}
\end{minipage}
\end{figure}
Now we discuss the numerical computation of the smallest two roots of the transcendental equation (\ref{eq:trans}) and hence the gap $\Lambda(m \boldsymbol{1}_{(0,\pi/2)},\alpha)$. All computations were performed in Mathematica Version 12.0.0.0.

Our proof of Theorem \ref{thm:Horvath} does not handle the case when $\alpha < 0$. This is because Lemma \ref{lem:Trig} fails to show that the gap $m \mapsto \Lambda(m \boldsymbol{1}_{(0,\pi/2)},\alpha)$ is minimal for $m = 0$ when $\alpha < 0$. Nonetheless, the numerical computation of $m \mapsto \Lambda(m \boldsymbol{1}_{(0,\pi/2)},\alpha)$ in Figures \ref{fig:NegGapM} and $\ref{fig:PosGapM}$ suggests the gap is actually monotone for all $m \geq 0$ and for each $\alpha \in (-\infty,\infty]$. This lends support to Conjecture \ref{conj:AllAlpha}.

We also plot the gap $\Lambda(m \boldsymbol{1}_{(0,\pi/2)},\alpha)$ as a function of $\alpha$ in Figure \ref{fig:GapAlphaPlot}. When $m$ is large the plot shows that $\alpha \mapsto \Lambda(m \boldsymbol{1}_{(0,\pi/2)},\alpha)$ is monotone only for $\alpha \geq 0$. This phenomenon appears to continue for small $m$ but is more subtle in this regime. This example shows there are asymmetric potentials for which the gap is monotone for $\alpha \geq 0$, but fails to be monotone for $\alpha < 0$. Hence, Open Problem \ref{op:mono} is only stated for $\alpha \geq 0$.

%%%%%%%%%%%%%%%%%%%%%%%%%%%%%%%%%%%%%%%%%%%%%%%%%%%%%%%%%%%%%%%%%%%%%%%%%%%%%%%%%%%%%%%%%%%%%%%%%%%%%%%%%%%%%%%%%%%%%%%%%%%%%%%%%%%%%%%%%%

\section{\textbf{Symmetric single well --- proof of Theorem \ref{thm:SWS}}} 

In this section assume that $\beta = \alpha$. Define $V(x,t) =  S(x) + t V(x)$ for $t \in [0,1]$ so that $V(x,0) = S(x)$ and $V(x,1) = S(x) + V(x)$. Lemma \ref{lem:Wrskn} implies that there is a point $x_-(t) \in (-L/2,0)$ such that 
\begin{gather}
\label{eq:SSWineq}
\begin{split}
u_2(x,t)^2 > u_1(x,t)^2 ~ \text{for} ~ x \in (-L/2,x_-(t)) \cup (- x_-(t),L/2), ~ \text{and}\\
u_2(x,t)^2 < u_1(x,t)^2 ~ \text{for} ~ x \in (x_-(t),-x_-(t)) ~ \text{for each} ~ t \in [0,1].
\end{split}
\end{gather}
To prove the theorem first we compute derivatives of the gap in $t$ for fixed $\alpha$ using Lemma \ref{lem:Deriv} to find
\[\frac{d\Lambda}{dt}(t) = \int_I \frac{\partial V}{\partial t} (u_2^2 - u_1^2) \, dx = \int_I V(x)(u_2(x)^2 - u_1(x)^2) \, dx.\]
Since $u_2^2 - u_1^2$ has mean zero we can replace $V(x)$ by $V(x) - V(x_-(t))$ in the integral on the right. Since $V$ is symmetric single-well and (\ref{eq:SSWineq}) holds it follows that the integral is non-negative and zero if and only if $V$ is constant. 

We can compute the $\alpha$-derivative of the gap for fixed $t$ and $\alpha \in \R$
\[\frac{d\Lambda}{d \alpha} = u_2(-L/2)^2 + u_2(L/2)^2 - (u_1(-L/2)^2 + u_1(L/2)^2),\]
which is non-negative by (\ref{eq:SSWineq}). To see that it is positive it suffices to show that $u_2(-L/2)^2 > u_1(-L/2)^2$ by symmetry. We have $u_2(-L/2)^2 > u_1(-L/2)^2$ since the proof of Lemma \ref{lem:Wrskn} shows that $(u_2/u_1)(x)$ is decreasing and strictly larger than $1$ at $x = -L/2$. This proves $\frac{d \Lambda}{d \alpha} > 0$ for all $\alpha \in \R$. Since $\Lambda(S + tV,\alpha)$ is defined and continuous for $\alpha \in (-\infty,\infty]$ it follows that it is strictly increasing on $(-\infty,\infty]$. \qed

\begin{rem}
Theorem \ref{thm:SWS} can also be proved by adapting the trial function technique used by Ashbaugh and Benguria in \cite{AB}. This technique has the advantage that it gives a lower bound for the difference $\Lambda(S + V, \alpha + \beta) - \Lambda(S,\alpha)$ in terms of an integral of a combination of the relevant eigenfunctions. At the moment it is unclear how to use this to produce a lower bound that is explicit.
\end{rem}

%%%%%%%%%%%%%%%%%%%%%%%%%%%%%%%%%%%%%%%%%%%%%%%%%%%%%%%%%%%%%%%%%%%%%%%%%%%%%%%%%%%%%%%%%%%%%%%%%%%%%%%%%%%%%%%%%%%%%%%%%%%%%%%%%%%%%%%%%%

\section{\textbf{Convex potentials}}

The proof Theorem \ref{thm:VConvNew} proceeds by first proving that gap has a minimizer in $\C$  for appropriate choices of $M,\underline{\alpha} = \min\{\alpha,\beta\},\overline{\alpha},\underline{\beta},$ and $\overline{\beta}$. Following this we show that the minimizer is a potential-boundary condition pair of the form $(ax,(\underline{\alpha},\underline{\alpha}))$. To conclude we use that among linear potentials with fixed symmetric Robin parameters the gap is minimized by the constant potential (see \cite[Theorem 1.3]{ACH2}). The following lemma can be found in the work of Andrews, Clutterbuck, and Hauer \cite[Lemma 5.1]{ACH2} for the case $\beta = \alpha$.

\begin{lemma}
\label{lem:u1u2}
Assume that $u_1$ and $u_2$ are $L^2$-normalized first and second eigenfunctions of $-\frac{d^2}{dx^2} + ax$ with Robin boundary conditions $(\alpha,\beta) \in \R^2 $. If $a$ is such that $\int_0^R x(u_2^2 - u_1^2) \, dx = 0$ then $u_2(-L/2)^2 > u_1(-L/2)^2$ and $u_2(L/2)^2 > u_1(L/2)^2$.
\end{lemma}

\begin{proof}
Suppose, for the sake of contradiction, that either $u_2(-L/2)^2 \leq u_1(-L/2)^2$ or $u_2(L/2)^2 \leq u_1(L/2)^2$. Without loss of generality, we may assume the latter inequality holds by reflecting the potential and swapping $\alpha$ with $\beta$, if necessary. It follows that $x_- \in (-L/2,L/2)$ and $x_+ = L/2$. Since $u_2^2 - u_1^2$ has mean zero we have that
\[0 = \int_{-L/2}^{L/2} x(u_2^2 - u_1^2) \, dx = \int_{-L/2}^{L/2} (x - x_-)(u_2^2 - u_1^2) \, dx < 0,\]
since $u_2^2 - u_1^2 > 0$ on $[-L/2,x_-)$ and $u_2^2 - u_1^2 < 0$ on $(x_-,L/2]$. This is a contradiction so we must conclude that $u_2(-L/2)^2 > u_1(-L/2)^2$ and $u_2(L/2)^2 > u_1(L/2)^2$.
\end{proof}

\subsection*{Proof of Theorem \ref{thm:VConvNew}}

Since the gap $\Lambda(V,(\alpha,\beta))$ is invariant under adding a constant to $V$ we can assume that $V \geq 0$.

First assume that $\alpha,\beta \in [-1/L,\infty)$, that is, $\alpha,\beta \neq \infty$. To prove the inequality we first note that there is a $(V_0,\alpha_0,\beta_0) \in \C$ (see (\ref{eq:ConvSet}) for a definition) that minimizes $\Lambda(V,(\alpha,\beta))$ by taking $\underline{\alpha} = \underline{\beta} = \min\{\alpha,\beta\}, \overline{\alpha} = \overline{\beta} = \max\{\alpha,\beta\},$ and $M = \max_I\{V\}$ in Lemma \ref{lem:MinEx}. Next we argue that a minimizer is of the form $(ax+b,(\underline{\alpha},\underline{\alpha}))$ for some $a,b \in \R$. To do so, first we make a perturbation of the potential to conclude that it must be linear and then we make a perturbation of the boundary parameters to conclude that they must be as small as possible.

Consider the family of convex potentials
\[V(x,t) = (1 - t) V_0(x) + t V_1(x), \quad \text{ for each } t \in [0,1],\]
where $V_1(x)$ is the linear function with graph that intersects the coordinates $(x_-,V_0(x_-))$ and $(x_+,V_0(x_+))$ in $\R^2$. Using Lemma \ref{lem:Deriv} we can compute the derivative of the gap and use minimality of $V_0$ to find
\[\frac{d \Lambda}{d t} \bigg|_{t  = 0} = \int_I \frac{dV}{dt} \bigg|_{t = 0} (u_2^2 - u_1^2) \, dx \geq 0.\]
Using that $\frac{dV}{dt}|_{t = 0} = V_1(x) - V_0(x) \geq 0$ on $(x_-,x_+)$ and non-positive on the complement we have that $V_0(x) = V_1(x) = ax + b$ for each $x \in I$ with $a,b \in \R$ by using Lemma \ref{lem:Wrskn}.
Since the eigenvalues are invariant under adding a constant to the potential we may assume that $V_0(x) = ax$ for some $a \in \R$. 

Minimality of $V_0$ also implies that
the derivative of the gap with respect to $a$ is zero so that 
\[0 = \frac{d}{da} \Lambda(ax,(\alpha_0,\beta_0)) = \int_I x(u_2^2 - u_1^2) \, dx.\] 
It follows from Lemma \ref{lem:u1u2} that $u_2(\pm L/2)^2 > u_1(\pm L/2)^2$ so that Lemma \ref{lem:Deriv} implies we can decrease the larger of the two Robin parameters while decreasing the gap. By minimality of the gap we conclude that $\alpha_0 = \min\{\alpha,\beta\} = \beta_0$. Thus, $\Lambda(ax,(\alpha_0,\beta_0)) \geq \Lambda(ax,\underline{\alpha})$.

To complete the proof we use that among linear potentials the gap is minimized by the constant one \cite[Theorem 1.3]{ACH2} so that $\Lambda(V_0,\underline{\alpha}) \geq \Lambda(0,\underline{\alpha})$ since $\underline{\alpha} \geq -1/L$. Thus, we have proved
\[\Lambda(V,(\alpha,\beta)) \geq \Lambda(0,\underline{\alpha}).\]

To include the possibility that $\alpha$ or $\beta$ (or both) is infinity, we note that, for example, if $\alpha = \infty$ and $\beta < \infty$ then the inequality continues to hold since $\lambda_j(V,(\alpha,\beta)) \to \lambda_j(\infty,\beta)$ as $\alpha \to \infty$. Similar observations allow us to conclude the proof in all cases. \qed

%%%%%%%%%%%%%%%%%%%%%%%%%%%%%%%%%%%%%%%%%%%%%%%%%%%%%%%%%%%%%%%%%%%%%%%%%%%%%%%%%%%%%%%%%%%%%%%%%%%%%%%%%%%%%%%%%%%%%%%%%%%%%%%%%%%%%%%%%%

\section*{\textbf{Acknowledgements}}
Derek Kielty gratefully acknowledges support from the University of Illinois Campus Research Board award RB19045 (to Richard Laugesen) and travel support from the University of Missouri and the University of Illinois at Urbana-Champaign. We would also like to thank the American Institute of Mathematics and specifically Evans Harrell, David Krej\v ci\v r\'ik, and Vladimir Lotoreichik for hosting and organizing the ``Shape optimization with surface interactions" workshop \cite{AIM2} where our collaboration began. We also thank Richard Laugesen for helpful comments and suggestions as this research progressed.

\bibliographystyle{plain}
\bibliography{refs}

\begin{thebibliography}{10}

\bibitem{A}
S.~Abramovich.
\newblock The gap between the first two eigenvalues of a one-dimensional
  {S}chr\"odinger operator with symmetric potential.
\newblock {\em Proc. Amer. Math. Soc.}, 111(2):451--453, 1991.

\bibitem{AIM1}
A{merican Institute of Mathematics workshop: Low eigenvalues of Laplace and
  Schr\"odinger operators. 2006.
  \url{https://aimath.org/pastworkshops/loweigenvalues.html}}.

\bibitem{AIM2}
A{merican Institute of Mathematics workshop: Shape optimization with surface
  interactions. 2019.
  \url{https://aimath.org/pastworkshops/shapesurface.html}}.

\bibitem{AC}
B.~Andrews and J.~Clutterbuck.
\newblock Proof of the fundamental gap conjecture.
\newblock {\em J. Amer. Math. Soc.}, 24(3):899--916, 2011.

\bibitem{ACH2}
B.~Andrews, J.~Clutterbuck, and D.~Hauer.
\newblock The fundamental gap for a one-dimensional {S}chr\"odinger operator
  with {R}obin boundary conditions.
\newblock {\em \arxiv{2002.06900}}.

\bibitem{AshAIM}
M.~S. Ashbaugh.
\newblock The fundamental gap.
\newblock {\em \url{https://www.aimath.org/WWN/loweigenvalues/gap.pdf}}.

\bibitem{AB}
M.~S. Ashbaugh and R.~Benguria.
\newblock Optimal lower bound for the gap between the first two eigenvalues of
  one-dimensional {S}chr\"odinger operators with symmetric single-well
  potentials.
\newblock {\em Proc. Amer. Math. Soc.}, 105(2):419--424, 1989.

\bibitem{BB}
R.~Ba{\~n}uelos and K.~Burdzy.
\newblock On the ``hot spots" conjecture of {J}. {R}auch.
\newblock {\em J. Funct. Anal.}, 164(1):1--33, 1999.

\bibitem{BT}
B.~Buffoni and J.~Toland.
\newblock {\em Analytic Theory of Global Bifurcation}.
\newblock Princeton University Press, 2003.

\bibitem{CKLL}
Y.~H. Cheng, S.~Y. Kung, C.~K. Law, and W.~C. Lian.
\newblock The dual eigenvalue problems for the {S}turm--{L}iouville system.
\newblock {\em Comput.\ Math.\ Appl.}, 60(9):2556--2563, 2010.

\bibitem{HZ}
E.~Harrell and Z.~El Allali.
\newblock Optimal bounds on the fundamental spectral gap with single-well
  potentials.
\newblock {\em \arxiv{1807.08328}}.

\bibitem{Henrot}
A.~Henrot, editor.
\newblock {\em Shape Optimization and Spectral Theory}.
\newblock Berlin, Boston: De Gruyter, 2017.

\bibitem{Horvath}
M.~Horv\'ath.
\newblock On the first two eigenvalues of {S}turm--{L}iouville operators.
\newblock {\em Proc. Amer. Math. Soc.}, 131(4):1215--1224, 2002.

\bibitem{L}
R.~S. Laugesen.
\newblock The {R}obin {L}aplacian -- spectral conjectures, rectangular
  theorems.
\newblock {\em Journal of Mathematical Physics}, 60, 2019.

\bibitem{Lav}
R.~Lavine.
\newblock The eigenvalue gap for one-dimensional convex potentials.
\newblock {\em Proc. Amer. Math. Soc.}, 121(3):815--821, 1994.

\bibitem{PW}
L.~E. Payne and H.~F. Weinberger.
\newblock An optimal {P}oincar\'e inequality for convex domains.
\newblock {\em Arch. Rational Mech. Anal.}, 5:286--292, 1960.

\bibitem{RS}
Michael Reed and Barry Simon.
\newblock {\em Methods of Modern Mathematical Physics. {IV}. {A}nalysis of
  operators}.
\newblock Academic Press [Harcourt Brace Jovanovich, Publishers], New
  York-London, 1978.

\bibitem{Smits}
R.~G. Smits.
\newblock Spectral gaps and rates to equilibrium for diffusions in convex
  domains.
\newblock {\em Michigan Math. J.}, 43:141--157, 1996.

\bibitem{Thirring}
W.~Thirring.
\newblock {\em A Course in Mathematical Physics. {V}ol. 3. Quantum Mechanics of
  Atoms and Molecules}.
\newblock Springer-Verlag, New York-Vienna, 1981.
\newblock Translated from the German by Evans M. Harrell.

\bibitem{Widom}
H.~Widom.
\newblock {\em Lectures on Integral Equations}.
\newblock Van Nostrand-Reinhold, New York, 1969 (reprinted by Dover
  Publications, 2017).

\end{thebibliography}

\end{document}